\documentclass{aptpub}
\usepackage{graphicx}
\usepackage[normalem]{ulem}
\usepackage{bookmark}
\usepackage{color}
\usepackage{amsmath}
\DeclareMathOperator{\dist}{dist}

\authornames{Feng Zhao {\it et al.}} 
\shorttitle{On the expected number of facets for the convex hull of samples} 

\def\E{\mathbb{E}}
\def\R{\mathbb{R}}
\def\d{\mathrm{d}}
\begin{document}

\title{On the expected number of facets for the convex hull of samples from spherically symmetric distributions} 

\authorone[Department of Electronics, Tsinghua University]{Feng Zhao} 

\addressone{Tsinghua University} 
\emailone{zhaof17@mails.tsinghua.edu.cn} 

\authortwo[Tsinghua Berkeley Shenzhen Institute]{Xinyi Tong}
\addresstwo{Tsinghua Berkeley Shenzhen Institute}
\emailtwo{txy18@mails.tsinghua.edu.cn}

\authorthree[Tsinghua Berkeley Shenzhen Institute]{Shao-Lun Huang}
\addressthree{Tsinghua Berkeley Shenzhen Institute}
\emailtwo{shaolun.huang@sz.tsinghua.edu.cn}

\begin{abstract}
This paper studies the convex hull of $d$-dimensional samples i.i.d. generated from spherically symmetric distributions. Specifically, we derive a complete integration formula for the expected facet number of the convex hull.
This formula is
with respect to the CDF of the radial distribution.
 As the number of samples approaches infinity,
 the integration formula enables us to obtain the asymptotic value of the expected facet number
 for three categories of spherically symmetric distributions.
 Additionally, the asymptotic result can be applied to estimating the sample complexity in order that
 the probability measure of the convex hull
 tends to one.
\end{abstract}



\section{Introduction}\label{sec:intro} 


Let $X_1, X_2, \dots, X_N$ be i.i.d. random points generated from
a spherically symmetric distribution in $\mathbb{R}^d$.
For the convex hull $\mathrm{H}_N$ of these $N$ points, we study the number of its facets $F_N$,
where a facet of a $d$-dimensional object is one of its $(d-1)$-dimensional faces.
In this paper, we study the mathematical expectation
of $F_N$, denoted as $\E[F_N]$, and derive its asymptotic value as $N\to \infty$.

Our motivation to study $\E[F_N]$ is to give an upper bound of
the probability $p_{N,d}=P(X_{N+1} \not\in \mathrm{H}_N)$,
which is the probability that the $(N+1)$-th point falls outside the convex hull.
If we accept the concept that interpolation occurs when $X_{N+1}$ belongs to $H_N$,
$p_{N,d}$ can be applied to explain
why interpolation almost surely never occurs in high dimensional space \cite{balestriero2021learning}.
In other words, for a large $d$, $p_{N,d}$ is near $1$, unless exponentially large sample size $N$ is available.
Using the asymptotic result on $\E[F_N]$, we obtain
a sufficient condition under which the interpolation almost surely occurs.

The asymptotic expression of $\E[F_N]$ as $N\to \infty$
was studied initially by R{\'e}nyi and Sulanke \cite{renyi1963konvexe}.
They considered
bivariate Gaussian distribution and uniform distribution
in a planar convex region.
Later on,  R{\'e}nyi's work was generalized by
Carnal \cite{carnal1970konvexe}, who
classified symmetric 2-D distributions
into three categories according to their tails:
polynomial, exponential, and truncated tails.
Then he obtained the asymptotic expression of $\E[F_N]$
for each category of distributions.

The study of $\E[F_N]$ for $d>2$ was made firstly by
Raynaud
\cite{raynaud1970enveloppe}.
He obtained the asymptotic formula of $\E[F_N]$
for uniform distribution in a hyperball
and standard Gaussian distribution in $\mathbb{R}^d$.
Afterwards, following Carnal, Dwyer \cite{dwyer1991convex}
estimated the order of $\E[F_N]$ about $N$
for three different distribution families.

Dwyer's work on the estimation of $\E[F_N]$ only captured its relationship with $N$ but ignored its
dependence on $d$. To our best knowledge, previous works have not revealed the general asymptotic expression of $\E[F_N]$ in $\R^d$.
In Section \ref{sec:int_f}, we develop a method of obtaining $\E[F_N]$ by generalizing Carnal's integration formula
from $d=2$ to higher dimensions.
In Section \ref{sec:three_distriutions}, we then derive the explicit asymptotic
expressions of $\E[F_N]$ for three different distribution families.
In Section \ref{sec:sample_complexity}, based on the asymptotic result of $\E[F_N]$,
we provide a sufficient condition for the convergence of $p_{N,d}$ to zero as $N,d \to \infty$.
The main contribution of this paper is to provide a complete integration formula for $\E[F_N]$,
and following the classification of distributions of previous authors \cite{carnal1970konvexe,dwyer1991convex},
we derive the asymptotic expression
of $\E[F_N]$ for each category.

Below we define some notations to be used throughout this paper.
Without specific emphasis, all $d$-dimensional distributions considered in this paper are spherically symmetric.
Let $X=(X^{(1)},\dots, X^{(d)})$ follow such a distribution.
Then we use $F(x)=P(|X|\geq x)$ to denote the probability that a random point lies outside
a $d$-sphere ($d$-dimensional sphere centered at origin) with radius $x$.
The marginal distribution of each component of $X$ is determined by $G(x)=P(X^{(1)}\geq x)$.
Let $\dist$ represent the distance from the origin to the hyperplane spanned by $X_1, \dots, X_d$.
Specifically, in 2-D, the hyperplane is the straight line passing through $X_1$ and $X_2$.
We use $H(x)=P(\dist\geq x)$ to denote the probability that $\dist$ is larger than $x$.
When $n$ tends to infinity, the functions $f(n)$ and $g(n)$ are asymptotic equivalent if $\lim_{n\to \infty} \frac{f(n)}{g(n)}=1$.
We denote this relationship as $f(n) \sim g(n)$.

\section{Related work}
Below we list some other related works which are not mentioned in the previous section.
For $d=2,3$, Efron \cite{efron1965convex} obtained the explicit integration formula for $\E[F_N]$ and $p_{N,d}$ 
when samples follow Gaussian distribution or uniform distribution within a unit sphere.

Davis et al. \cite{davis1987convex} found the relationship between the distributions with algebraic tails $F(x) \sim x^{-k}$ and Poisson random process.
They showed that for $k>0, d=2$,
$\mathrm{H}_N$ converges to the convex hull of a Poisson random process,
and
the limit of $\E[F_N]$ can be derived from this random process.
The special case $k=0, d=2$ is treated in a later paper, in which the
limit distribution of $F_N$ is computed \cite{aldous1991number}.

Studies related to distributions with exponential tails mainly focus on Gaussian distribution.
Kabluchko et al. \cite{kabluchko2020absorption} obtained explicit expressions for $p_{N,d}$;
Affentranger \cite{affentranger1991convex} derived $\E[F_N]$;
Hueter et al. \cite{hueter1999limit} studied the concentration property of $F_N$ and obtained
an upper bound for $\E[F_N]$ sharper than the bound in \cite{dwyer1991convex}.

For truncated tails,
Affentranger \cite{affentranger1991convex} studied one sub-category called beta-typed distribution and obtained
the asymptotic value of $\E[F_N]$.

Besides the number of facets $\E[F_N]$,
the asymptotic value of other quantities related with $\mathrm{H}_N$, such as the volume, and the surface area,
were systematically investigated in a general framework involving the property of facets
\cite{schneider2008stochastic, barany2008random}.

 \section{Integration formula}\label{sec:int_f}
In this section, we give the formula of $\E[F_N]$ with respect to the radial function $F(x)$.

Carnal \cite{carnal1970konvexe}
obtained
this formula for 2-D distributions,
which is given as follows:
\begin{align}
     \E[F_N] &= \binom{N}{2} \int_0^{\infty} 
     \left[G(x)^{N-2} + (1-G(x))^{N-2} \right]|\mathrm{d} H(x)| 
     \label{eq:E_F_N_2_d}
\end{align}
The integration formula of $\E[F_N]$ involves the function $G(x)$ and $H(x)$,
whose definitions are given at the end of Section \ref{sec:intro}.
Both functions are expressed with $F(x)$ in the following way:
\begin{align}
   G(x) &=\frac{1}{\pi} \int_x^{\infty}\arccos\frac{x}{y} |\mathrm{d} F(y)| \\
     H(x) &= \frac{2}{\pi} \int_x^{\infty} \arccos \frac{x}{y} |\mathrm{d}(F^2(y))|
     \label{eq:H_expression_2_dim}
\end{align}

For $d\geq 2$, the integration formula is generalized as
\begin{align}
     \E[F_N] &= \binom{N}{d} \int_0^{\infty} 
     \left[G(x)^{N-d} + (1-G(x))^{N-d} \right]|\mathrm{d} H(x)| 
     \textrm{ for } N \geq d+1 \label{eq:E_F_N_d}
\end{align}
Formula \eqref{eq:E_F_N_d} first appeared in the proof section of \cite{raynaud1970enveloppe}
and can further be generalized to compute other combinatorial properties of $\mathrm{H}_N$ (like surface area and volume)
\cite{barany2008random}.
Dwyer obtained the integration formula for $G(x)$ in $d\geq 2$ as:
\begin{align}\label{eq:G_d_kappa}
     G(x) & = \int_x^{+\infty} \kappa \left(\frac{x}{y} \right) |\mathrm{d}F(y)| \\
     \kappa(r) & = \frac{\Gamma(\frac{d}{2})}
     {\sqrt{\pi}\Gamma(\frac{d-1}{2})}\int_r^{1}
     (1-u^2)^{(d-3)/2}\mathrm{d}u\label{eq:kappa_r}
\end{align}
where $\kappa(r)$ is the fraction of the surface area of a unit $d$-sphere
cut off by a plane at distance $r$ from the origin.

We cannot compute $\E[F_N]$ from $F(x)$ yet, since no formula for $H(x)$ with respect to $F(x)$ is given in previous studies.
Our main theorem solves this problem by giving the integration formula for $H(x)$ in $d\geq 2$:
\begin{theorem}\label{thm:H}
For $d\geq 2$, the integration formula for $H(x)$ is given as
\begin{equation}
     H(x) = \frac{2}{\pi}
     \int_x^{+\infty} \arccos\frac{x}{y}
     |\mathrm{d} (K^d(y))|\label{eq:H_expression_d_dim}
\end{equation}
where the auxiliary function $K(x)$ is defined in the following way:
\begin{align}
     \lambda_d(x) & :=(1-x^2)^{\frac{d-2}{2}}
     \label{eq:lambda_r}\\
     K(x) &:=P\left(\sqrt{(X^{(1)})^2+(X^{(2)})^2}>x \right)=
     \int_x^{+\infty}
     \lambda_d \left(\frac{x}{y} \right)|\d F(y)|
     \label{eq:K_x}
\end{align}
\end{theorem}
It is easy to observe that \eqref{eq:H_expression_d_dim} reduces to 
\eqref{eq:H_expression_2_dim} when $d=2$.
For standard Gaussian distribution in $\R^d$,
$X^{(1)}, X^{(2)}$ are independent
Gaussian random variables, and we can verify that \eqref{eq:K_x} holds with $K(x) = e^{-x^2/2}$.
In fact, $1-K(x)$ is the CDF of Rayleigh distribution.
The proof of Theorem \ref{thm:H}
is provided in Appendix \ref{app:th}.

\section{Three distribution families}\label{sec:three_distriutions}
Notice that in \eqref{eq:E_F_N_d}, $G(x)<\frac{1}{2}$ for $x>0$. Therefore the
term $G(x)^{N-d}$ decays at an exponential rate as $N-d\to \infty$.
As a result, we obtain
\begin{align}
     \E[F_N] \sim \binom{N}{d} \int_0^{+\infty} 
      (1-G(x))^{N-d} |\d H(x)| \textrm{ as } N-d\to \infty
     \label{eq:E_F_N_d_sim}
\end{align}
Let $\E[V_N]$ represent the expected number of
vertices of $\mathrm{H}_N$.
By far it is not clear whether there exist explicit relations
between $\E[F_N]$ and $\E[V_N]$
for general spherical symmetric distributions.
However, we can give some estimations by inequality.
Using Corollary 19.6 of \cite{brondsted2012introduction}, we have the following
inequality:
\begin{equation}\label{eq:F_V_upper}
     F_N \geq (d-1) V_N - (d+1)(d-2)
 \end{equation}
Combined with $p_{N,d} = \frac{\E[V_{N+1}]}{N+1}$, which comes from
\cite{efron1965convex}, we have the following upper bound for $p_{N,d}$:
\begin{equation}\label{eq:p_N_d_bound}
    p_{N,d} \leq \frac{\E[F_N]}{d N} \textrm{ as } N \gg d
\end{equation}

We use the symbol $a \gg b$ if $a$ is a function of $b$ and $\lim_{b\to \infty} \frac{a}{b} = \infty$.
Based on \eqref{eq:E_F_N_d_sim}, in the following
we derive the asymptotic expression of $\E[F_N]$
for three different distribution families.

To precisely define these distribution families, we introduce the concept of slowly varying function.
\begin{definition}
A function $L(x)$ is
slowly varying as $x\to \infty$
if for all $\lambda>0$,
$\lim_{x\to\infty}\frac{L(\lambda x)}{L(x)}=1$
holds.
\end{definition}

\subsection{Distributions with polynomial tails}

With the help of a slowly varying function $L(x)$,
distributions with polynomial tails are written
in the following form:
\begin{equation}\label{eq:F_poly_tail}
     F(x) = x^{-k} L(x), k\geq 0
\end{equation}

Dwyer\cite{dwyer1991convex} has obtained $G(x)$ as:
\begin{equation}\label{eq:g_poly_tail}
     G(x) \sim \frac{\Gamma(\frac{d}{2})}{2\sqrt{\pi} \Gamma(\frac{d-1}{2})}
     B\left(\frac{k+1}{2}, \frac{d-1}{2}\right) F(x)  \textrm{ as } x\to \infty
\end{equation}

From Theorem \ref{thm:H}, we obtain the asymptotic
expression of $H(x)$ and $\E[F_N]$, which is
summarized in the following theorem:
\begin{theorem}\label{thm:poly_tails}
     For distributions with polynomial tails defined in \eqref{eq:F_poly_tail},
     we have
\begin{equation}\label{eq:H_poly_tail_exp}
     H(x) \sim \frac{2^d \pi^{(d-1)/2}\Gamma^d(\frac{k}{2}+1)
     \Gamma(\frac{dk+1}{2})}{
         \Gamma^d(\frac{k+1}{2}) \Gamma(\frac{dk}{2}+1)} G(x)^d 
         \textrm{ as } x\to \infty
\end{equation}
and 
\begin{equation}\label{eq:efn_poly_second_deri}
    \E[F_N] \sim \frac{2^d \pi^{(d-1)/2}\Gamma^d(\frac{k}{2}+1)
    \Gamma(\frac{dk+1}{2})}{
        \Gamma^d(\frac{k+1}{2}) \Gamma(\frac{dk}{2}+1)}
        \textrm{ as } N \to \infty, d \textrm { is fixed}
\end{equation}
\end{theorem}
Equation \eqref{eq:efn_poly_second_deri} tells
us that $\E[F_N]$ converges to a constant as $N \to \infty$.
Without consideration of the actual constant, the result of
Theorem \ref{thm:poly_tails} is an refinement to Dwyer's Theorem 1 for the expected number of facets \cite{dwyer1991convex}.
When $d=2$, the result was obtained in (2.4) of \cite{carnal1970konvexe}
and Theorem 4.4 of \cite{davis1987convex}.
\begin{example}
     We consider a special case of polynomial tail called multivariate t-distribution.
     Its pdf is given by
     \begin{equation}\label{eq:pxy_student_t}
          p(x) = \frac{\Gamma((k+d)/2)}{\Gamma(k/2)(k\pi)^{d/2}}
          \left(1+\frac{1}{k}||x||^2
          \right)^{-\frac{k+d}{2}}, x \in \mathbb{R}^d
      \end{equation}
      The parameter $k$ is the degree of freedom for this distribution.
      We can obtain the radius distribution
      $F(x)$ as $F(x) \sim \frac{2\Gamma(\frac{k+d}{2})}{\Gamma(k/2)\Gamma(d/2)} k^{k/2-1} x^{-k}$.
      The marginal distribution is Student's t-distribution. To obtain $G(x)$, which is just the tail area
of the t-distribution, we use an existing asymptotic result found in \cite{andrew1976}.
\begin{equation} \label{eq:eq_dv}
    G(x) \sim k^{\frac{k}{2}-1} \frac{\Gamma \left(\frac{k+1}{2} \right)}
    {\sqrt{\pi} \Gamma\left(\frac{k}{2}\right)}x^{-k}
\end{equation}
After simplification, \eqref{eq:eq_dv} is the same as \eqref{eq:g_poly_tail}.
On the other hand, using similar techniques as (1.4) in \cite{raynaud1970enveloppe},
we can obtain the pdf of the distance of the hyperplane to the origin as
\begin{equation}
     \frac{\d H(x)}{\d x} =  \frac{2}{\sqrt{k} B(\frac{kd}{2},\frac{1}{2})} \left(1 + \frac{x^2}{k} \right)^{-\frac{kd+1}{2}} 
\end{equation}
From above, we get the asymptotic relation $H(x) \sim \frac{2 \Gamma(\frac{kd+1}{2}) k^{kd/2-1}}{d\sqrt{\pi} \Gamma(\frac{kd}{2})}
x^{-kd}$ as $x\to \infty$, which is equivalent to \eqref{eq:H_poly_tail_exp}.

\end{example}
When we allow $d\to \infty$, we treat $N$ as a function of $d$.
If the condition $N/d^2 \to \infty$ is satisfied, we have
\begin{equation}\label{eq:poly_E_F_N_d_infty}
\E[F_N] \sim \sqrt{\frac{2}{\pi dk}}\left(
     \frac{\sqrt{\pi}k \Gamma(k/2)}
     {\Gamma(\frac{k+1}{2})}
 \right)^d \textrm{ for } k>0
\end{equation}

\subsection{Distributions with exponential tails}
In this subsection, we consider distributions satisfying
$x = L(1/F(x)) $, where $L(x)$ is slowly varying. Following Carnal \cite{carnal1970konvexe},
we define the following utility functions:
\begin{align}
     \epsilon(s) & = s (\log (L(s)))' \label{eq:epsilon_s}\\
     v(u) &= -\frac{1}{u} \frac{1}{(\log F(u))'}    
\end{align}
The prime notation represents the symbol for differentiation. When $v(u)$ satisfies
certain regularity conditions prescribed in (2.15) of \cite{carnal1970konvexe},
we can obtain that $\epsilon(s)$ is a slowly varying function.
Dwyer \cite{dwyer1991convex} obtains the expression of $G(x)$ as:
\begin{equation}\label{eq:G_x_exp}
     G(x) \sim \frac{2^{(d-3)/2}}{\sqrt{\pi}}\Gamma\left(\frac{d}{2}\right)
     v^{(d-1)/2}(x) F(x)
      \textrm{ as } x\to \infty
\end{equation}
We make the analysis complete by giving the following theorem.
\begin{theorem}\label{thm:exponential_tails}
     Suppose there exists a slowly
     varying function $L(x)$ such that \linebreak $x=L(1/F(x))$,
     and $\epsilon(s)$ is defined in \eqref{eq:epsilon_s}, then
\begin{equation}\label{eq:H_x_exp}
     H(x) \sim \frac{\pi^{\frac{d-1}{2}} 2^{\frac{d+1}{2}}}{\sqrt{d}}v^{\frac{-d+1}{2}}(x)G^d(x)
     \textrm{ as } x\to \infty
\end{equation}
 and
 \begin{equation}\label{eq:exp_e_f_n}
     \E[F_N]\sim \frac{\pi^{\frac{d-1}{2}} 2^{\frac{d+1}{2}}}{\sqrt{d}} (\epsilon(N))^{-\frac{d-1}{2}}
     \textrm{ as } N \to \infty, d \textrm { is fixed}
 \end{equation}
\end{theorem}
 Using \eqref{eq:epsilon_s}, we note that the term
 $(\epsilon(N))^{-\frac{d-1}{2}}$
 is obtained in Dwyer's Theorem 2 \cite{dwyer1991convex}.
 Our improvement on the asymptotic value of $\E[F_N]$ is that we obtain its preceding term only involving $d$.
 When $d=2$, equation \eqref{eq:exp_e_f_n} is consistent with (2.20) of \cite{carnal1970konvexe}.
 For standard $d$-dimensional Gaussian distribution, we have $L(s)=\sqrt{2\log s}$.
 From \eqref{eq:epsilon_s}, $\epsilon(s) = (2\log s)^{-1}$. And we obtain from \eqref{eq:exp_e_f_n}
 that $\E[F_N]\sim \frac{2^d}{\sqrt{d}}(\pi \log N)^{\frac{d-1}{2}}$,
 which is consistent with (1.11) of \cite{raynaud1970enveloppe}.

 If $d\to\infty$ and $N/d^2\to \infty$, we also have the sample asymptotic value
\begin{align}\label{eq:d_infty_exp_E_F_N}
      \E[F_N]\sim \frac{\pi^{\frac{d-1}{2}} 2^{\frac{d+1}{2}}}{\sqrt{d}} \epsilon(N)^{\frac{-d+1}{2}}
\end{align}
\subsection{Distributions with truncated tails}
In this subsection, let $L(x)$ also represent a slowly varying function, and we consider the distribution which satisfies
the following conditions:
\begin{equation}\label{eq:F_truncated}
     F(1-x) \sim x^k L\left(\frac{1}{x} \right)  \text{ as } x \to 0^+, k> 0,
     F(x) = 0 \text{ for } x \geq 1
\end{equation}
In \eqref{eq:F_truncated}, $x \to 0^+$ means that $x\to 0$ from the right hand side of $0$.

As mentioned by Dwyer \cite{dwyer1991convex}, uniform distribution
in the unit ball satisfies \eqref{eq:F_truncated} with
$F(1-x) \sim d\cdot x$.

After simplification of (4.1) in \cite{dwyer1991convex}, we obtain
\begin{align}
    G(1-x) \sim a
    L\left(\frac{1}{x} \right)
    x^{k+\frac{d-1}{2}} \textrm{ as } x \to 0^+ 
    \label{eq:truncated_G_1_x}
\end{align}
where
\begin{align}
a &=\frac{2^{\frac{d-1}{2}} k \Gamma(\frac{d}{2})}
    {(d-1) \sqrt{\pi} \Gamma(\frac{d-1}{2})}
    B\left(k, \frac{d+1}{2}\right) \notag \\
    &= \frac{2^{\frac{d-3}{2}} k \Gamma(\frac{d}{2})\Gamma(k)}
    {\sqrt{\pi} \Gamma\left(k+\frac{d+1}{2}\right)}
    \label{eq:a}
\end{align}
Below, we give the theorem in regards to the asymptotic value of $H(1-x)$ and $\E[F_N]$:
\begin{theorem}\label{thm:truncated_tails}
     For distributions with truncated tails
     defined in \eqref{eq:F_truncated},
     we have
\begin{align}
     H(1-x)  \sim b
     L^d(1/x) x^{d(k+\frac{d}{2}-1)+\frac{1}{2}} 
     \textrm{ as } x \to 0^+ \label{eq:truncated_H_1_x}
\end{align}
where
\begin{align}
     b =  \frac{k^d}{\pi}
     2^{\frac{1}{2} + d(\frac{d}{2}-1)} B^d\left(k, \frac{d}{2}\right)
     B\left( \frac{1}{2},
     d\left(k+\frac{d}{2} -1 \right)+1 \right)
     \label{eq:b}
 \end{align}
When $N\to \infty$ and $d$ is fixed, we have
 \begin{align}\label{eq:efn_truncated_formula}
     \E[F_N] &\sim \frac{b}{d!}a^{-d+\frac{d-1}{2k+d-1}}
     \Gamma 
     \left(d+1-\frac{d-1}{2k+d-1}\right)
     N^{\frac{d-1}{2k+d-1}}
     L(N)
     ^{\frac{d-1}{2k+d-1}}
 \end{align}
 where the parameter $a$ is defined in \eqref{eq:a}.
\end{theorem}
 The leading polynomial term $N^{\frac{d-1}{2k+d-1}}$ in \eqref{eq:efn_truncated_formula}
 is the same with that in Dwyer's Theorem 3
 but other coefficients are missing in Dwyer's result
 \cite{dwyer1991convex}. Therefore, we say that Dwyer only made the order estimation
 of $\E[F_N]$ while we obtain its asymptotic value.
 When $d=2$, equation \eqref{eq:efn_truncated_formula} reduces to (3.4) of \cite{carnal1970konvexe}.
 For uniform distribution with $k=1, L(N)=d$, we can verify that
 \eqref{eq:efn_truncated_formula} is equivalent with (1.1)
 of \cite{raynaud1970enveloppe}.
 For beta-typed distribution with $k=q+1, L(N)=\frac{2^{q+1}}{(q+1)B(d/2,q+1)}$,
 we can verify that
 \eqref{eq:efn_truncated_formula} is equivalent with the expression of $c_3$
 in (3.3) of \cite{affentranger1991convex}.

 When we allow $d\to \infty$ and $N/d^2 \to \infty$, we obtain
 \begin{equation}\label{eq:truncated_d_inf}
  \E[F_N] \sim 2^{\frac{d+2k}{2}}\pi^{\frac{d-2}{2}} k\Gamma(k)e^k d^{\frac{d-3}{2}-k}
  N^{\frac{d-1}{2k+d-1}} L(N)^{\frac{d-1}{2k+d-1}}
 \end{equation}
\section{Sample complexity in data interpolation}\label{sec:sample_complexity}
In the field of machine learning,
Balestriero et al. \cite{balestriero2021learning}
argues that interpolation almost surely never occurs in high-dimensional spaces.
Therefore, we can not think machine learning algorithms work well because they can interpolate training data
well.
Following the idea of Balestriero, we can further enhance this argument using the probability $p_{N,d}$ introduced previously.
The randomly sampled points in $\R^d$ represent training data
while the interpolation
symbolizes the learning algorithm. Then $p_{N,d}$ represents
the probability of region not learnt by the algorithm.
When we expect $p_{N,d} \to 0$ for large $d$,
in this section
we find that we need
at least exponentially large $N$,
which is impractical for real-world dataset.

The sufficient condition for $p_{N,d} \to 0$ is summarized in Table \ref{tab:cond},
which are obtained
from the estimation inequality \eqref{eq:p_N_d_bound},
combined with the asymptotic expressions for three distribution families in \eqref{eq:poly_E_F_N_d_infty},
\eqref{eq:p_N_d_bound}, and
\eqref{eq:truncated_d_inf}.

\begin{table}[!ht]
     \centering
     \begin{tabular}{cc}
         \hline
         distribution family &  Condition for $p_{N,d} \to 0$ \\
         \hline
        Algebraic tails & $N \gg \frac{c^d}{d^{3/2}}, c=\frac{\sqrt{\pi}k\Gamma(k/2)}{\Gamma(\frac{k+1}{2})}>1$ \\
        \hline
        Exponential tails 
        & $ N\cdot \epsilon(N)^{(d-1)/2} \gg \frac{(2\pi)^{d/2}}{d^{3/2}}$ \\
        \hline
        Truncated tails 
        & $\frac{N^{2k/d}}{L(N)^{\frac{d-1}{2k+d-1}}}
        \gg (2\pi)^{d/2}d^{(d-5)/2}$\\
        \hline
     \end{tabular}
     \caption{Condition for $p_{N,d}\to 0$ when $d\to \infty$ and $N$ is a function of $d$}
     \label{tab:cond}
 \end{table}

For algebraic tails, we need $N\gg \frac{c^d}{d^{3/2}}$,
to guarantee $p_{N,d}\to 0$.  
For the special case of Gaussian distribution,
from $\epsilon(s)=(2\log s)^{-1}$ and the general formula for exponential tails in Table \ref{tab:cond},
we can estimate that
$N(d)$ grows faster than any $c^d$ (algebraic tails) but slower than $\exp(d^2)$.
For uniform distribution, from $L(N)=d$ and the general formula for truncated tails,
$N\gg \exp(c' d^2 \log d)$ for some constant $c'$.
Generally speaking, the sample complexity of $N$ is the smallest for algebraic tails and the largest
for truncated tails. In other words, the faster $F(x)$ decays, the larger $N$ becomes.

In conclusion, this section provides further theoretical justification for the argument that interpolation almost surely never occurs in high-dimensional spaces.
This is just one application area for the asymptotic result of $\E[F_N]$ we obtained in previous sections.
Other applications can be found in domains of complexity analysis of algorithms related with convex hulls
\cite{seidel1997}.





\appendix

\section{Proof of Theorem \ref{thm:H}}\label{app:th}
\begin{lemma}\label{lem:F_0}
     Let $F_0$ represent the
probability that the distance from the origin to the straight line
$X_1X_2$ is larger than $x$, where $X_1, X_2$ follow the distribution with $F(x)=P(|X|\geq x)$.
Then 
\begin{equation}\label{eq:F_0_expression}
     F_0(x)=\frac{2\Gamma(\frac{d}{2})}
     {\sqrt{\pi}\Gamma(\frac{d-1}{2})}
     \int_x^{\infty} |\d F(y)|
     \int_x^{y} |\d F(z)| \int_{a_2/z}^{a_1 /z} (1-u^2)^{\frac{d-3}{2}} \d u
 \end{equation}
where
\begin{align}
     a_1 & =\frac{x^2}{y}+\sqrt{z^2-x^2}\sqrt{1-\frac{x^2}{y^2}} > 0
     \label{eq:a_1} \\
a_2 & =\frac{x^2}{y}-\sqrt{z^2-x^2}\sqrt{1-\frac{x^2}{y^2}} < 0
\label{eq:a_2}
\end{align}
\end{lemma}
Lemma \ref{lem:F_0} gives the integration formula for $P(d_{12}\geq x)$ in $\R^d$.
The geometric meaning of $a_1, |a_2|$ is illustrated in Figure
\ref{fig:a1a2}. That is, if we assume
\begin{equation}\label{eq:theta_1_theta_2}
     \cos\theta_1=\frac{x}{z}
     \textrm{ and } \cos\theta_2=\frac{x}{y} 
\end{equation}
then
\begin{equation}\label{eq:a_1_a_2}
     \frac{a_1}{z} = \cos(\theta_2 - \theta_1),
     \quad
     \frac{a_2}{z} = \cos(\theta_2+\theta_1)           
\end{equation}

\begin{figure}[!ht]
     \centering
     \includegraphics[width=0.8\textwidth]{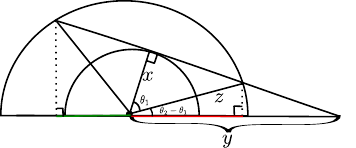}
     \caption{The length of red line represents $a_1$ while the green line corresponds to $|a_2|$.}
     \label{fig:a1a2}
\end{figure}
When $d=2$, using $\int \frac{\d x}{\sqrt{1-x^2}} = -\arccos x + C$,
we obtain
$$
F_0(x)=\frac{4}{\pi} \int_x^{\infty} \arccos\frac{x}{z}|\d F(y)|
\int_x^{y} |\d F(z)|\d z
$$
which is the same with \eqref{eq:H_expression_2_dim}.
\begin{proof}[Proof of Lemma \ref{lem:F_0}]
For $d\geq 2$ and $x<z<y$, we follow the geometric approach
which is adopted to derive
$H(x)$ in \cite{carnal1970konvexe}.
Then we only need to compute the ratio of surface area
on the sphere with radius $z$. This is the area between
two planes at distance $\frac{a_1}{z}$ and $\frac{a_2}{z}$
respectively on the unit sphere. From (1.1) of \cite{dwyer1991convex},
this ratio equals $\kappa(\frac{a_2}{z}) - \kappa(\frac{a_1}{z})$
where $\kappa$ is defined in \eqref{eq:kappa_r}.

\end{proof}
To show 
$K(x) = \int_x^{+\infty} \lambda_d(\frac{x}{y})|\d F(y)|$
in \eqref{eq:K_x},
we only need to connect $\lambda_d(r)$ with its geometric meaning,
which is given in the following lemma:
\begin{lemma}
     The ratio of surface area on a unit $d$-sphere,
     satisfying $x_1^2+x_2^2\geq r$ is $\lambda_d(r)$,
     defined in \eqref{eq:lambda_r}.
\end{lemma}
\begin{proof}
     Following \cite{dwyer1991convex}, we use $\kappa_d = 2\pi^{d/2}/\Gamma(d/2)$
     to represent the surface area of the unit $d$-sphere. Then
     \begin{align*}
          \lambda_d(r) &=\frac{1}{\kappa_d} 
          \cdot 2\int_{\substack{x_1^2+x_2^2\geq r^2\\
          x_1^2+\dots+x_{d-1}^2\leq 1 }} 
          \frac{1}{\sqrt{1-x_1^2-\dots -x_{d-1}^2}} \d x_1 \dots \d x_{d-1}\\
      &= \frac{4\pi}{\kappa_d} \int_r^1 x\d x \int_{x_3^2+\dots + x_{d-1}^2 \leq 1-x^2} \frac{\d x_3\dots \d x_{d-1}}
      {\sqrt{1-x^2-x_3^2-\dots -x_{d-1}^2}} \\
      &=\frac{4\pi \kappa_{d-3}}{\kappa_d} \int_r^1 x\d x \int_0^{\sqrt{1-x^2}} \frac{y^{d-4}\d y}{\sqrt{1-x^2-y^2}}\\
      &=\frac{4\pi \kappa_{d-3}}{\kappa_d} \int_0^{\sqrt{1-r^2}} y^{d-4}\d y \int_r^{\sqrt{1-y^2}} \frac{x\d x}{\sqrt{1-x^2-y^2}}\\
      &=\frac{4\pi \kappa_{d-3}}{\kappa_d} \int_0^{\sqrt{1-r^2}} y^{d-4}\sqrt{1-r^2-y^2} \d y\\
      &=\frac{4\pi \kappa_{d-3}}{\kappa_d}  \frac{B(\frac{3}{2 }, \frac{d-3}{2})}{2}(1-r^2)^{(d-2)/2}\\
      &=  (1-r^2)^{(d-2)/2}
      \end{align*}
\end{proof}

\begin{lemma}\label{lem:cn_integration}
For $0<c<1$ and $n$ is a non-negative positive integer, we have
\begin{equation}
    \int_0^{c}
    [\frac{1}{(1-t)^{n+1}}+\frac{1}{(1+t)^{n+1}}]
    (c^2- t^2)^{(n-1)/2}\d t
    =B(\frac{n+1}{2}, \frac{1}{2})
    \frac{c^n}{(1-c^2)^{(n+1)/2}}
    \end{equation}\label{eq:c_int_eq}
\end{lemma}
\begin{proof}[Proof of Lemma \ref{lem:cn_integration}]
     When $n=0$, the left hand side equals
     \begin{align*}
          \int_0^c \frac{2}{(1-t^2)\sqrt{c^2-t^2}} \d t
          = \int_0^{\sqrt{c}} \frac{1}{(1-t) \sqrt{t} \sqrt{c^2-t} }\d t
     \end{align*}     
     We make change of variables $x=\sqrt{t/(c^2-t)}$ and obtain
     the value $\frac{\pi}{(1-c^2)^{1/2}}$, which equals
     the right hand side.
     When $n=1$, we can use integration directly
     to show that \eqref{eq:c_int_eq} holds.

     Let
\begin{equation}\label{eq:f_n_c_def}
h(n,c)=   \int_0^{c}
    [\frac{1}{(1-t)^{n+1}}+\frac{1}{(1+t)^{n+1}}]
    (c^2- t^2)^{(n-1)/2}\d t
\end{equation}
When $n\geq 2$, by integration by parts 
we have 
\begin{equation}\label{eq:tmp_n_u_d_f_n}
    \frac{n}{n-1}h(n,c)
    =\int_0^{c}
    \left[\frac{1}{(1-t)^{n}}
    -\frac{1}{(1+t)^{n}}
    \right]
    t(c^2- t^2)^{(n-3)/2}
    \d t
\end{equation}
Let $t=\frac{(1+t)-(1-t)}{2}$ in the right hand side
of \eqref{eq:tmp_n_u_d_f_n}, then
\begin{equation}
    2\frac{n}{n-1}h(n,c)
=    -h(n-2,c)  
+ \int_0^{c}
\left[\frac{1+t}{(1-t)^{n}}
+\frac{1-t}{(1+t)^{n}}
\right]
(c^2- t^2)^{(n-3)/2}
\d t
\end{equation}
On the other hand,
From \eqref{eq:f_n_c_def},
$(c^2-t^2)^{(n-1)/2}
=[(c^2-1)+(1-t^2)](c^2-t^2)^{(n-3)/2}$
\begin{equation}
    h(n, c) = \frac{c^2-1}{2c}\frac{2}{n-1}\frac{\partial h(n,c)}{\partial c}
    +  \int_0^{c}
    \left[\frac{1+t}{(1-t)^{n}}
    +\frac{1-t}{(1+t)^{n}}
    \right]
    (c^2- t^2)^{(n-3)/2}
    \d t
\end{equation}
Therefore,
\begin{equation}\label{eq:sim_1tc}
    h(n,c)=\frac{c^2-1}{2c}\frac{2}{n-1}\frac{\partial h(n,c)}{\partial c}
    + [2\frac{n}{n-1} h(n,c) + h(n-2, c)]
\end{equation}
then we can use induction and solve this
first order linear ODE about $h(n,c)$.
The initial condition is $h(n,0)=0$ and
suppose
$h(n-2,c)=B(\frac{n-1}{2}, \frac{1}{2})
\frac{c^{n-2}}{(1-c^2)^{(n-1)/2}}$.
Then \eqref{eq:sim_1tc} is simplified as
\begin{equation}
    \frac{\partial h(n,c)}{\partial c}
    - \frac{(n+1)c}{1-c^2} h(n,c)
    = B(\frac{n-1}{2}, \frac{1}{2})\frac{(n-1)c^{n-1}}{(1-c^2)^{(n+1)/2}}
\end{equation}
Using the integration factor $I(c)=(1-c^2)^{(n+1)/2}$, we obtain
\begin{equation}\label{eq:f_n_c_expression}
    h(n,c)= B(\frac{n+1}{2}, \frac{1}{2})
    \frac{c^n}{(1-c^2)^{(n+1)/2}}
\end{equation}

\end{proof}

Having defined $K(x), F_0(x)$, we find the following relationship holds:
\begin{lemma}\label{lem:K_F_relationship}
\begin{equation}\label{eq:F_0_integration}
     \int_u^{+\infty}
     (1-\frac{u^2}{x^2})^{\frac{d-3}{2}} |\d F_0(x)| = K^2(u)
\end{equation}
\end{lemma}
\begin{proof}[Proof of Lemma \ref{lem:K_F_relationship}]
     By the chain rule, from \eqref{eq:a_1_a_2} we obtain
\begin{align*}
    \frac{\d}{\d x}\frac{a_1}{z} &
    = \sin(\theta_2 - \theta_1)
    \left(\frac{-1}{z\sin\theta_1}
    +\frac{1}{y\sin \theta_2}\right)\\
    \frac{\d}{\d x}\frac{a_2}{z} &
    = \sin(\theta_2 + \theta_1)
    \left(\frac{1}{z\sin\theta_1}
    +\frac{1}{y\sin \theta_2}\right)
\end{align*}
where $\theta_1, \theta_2$ are defined in \eqref{eq:theta_1_theta_2}.
Then from \eqref{eq:F_0_expression}
\begin{align*}
    \frac{\d F_0(x)}{\d x}  =\frac{2\Gamma(\frac{d}{2})}
    {\sqrt{\pi}\Gamma(\frac{d-1}{2})}
    \int_x^{+\infty} |\d F(y)| \int_x^y |f(x,y,z)| |\d F(z)| 
\end{align*}
where
\begin{align*}
    f(x,y,z) = \sin^{d-2} (\theta_2 - \theta_1)
    \left(\frac{-1}{z\sin\theta_1}
    +\frac{1}{y\sin \theta_2}\right)
    - \sin^{d-2}(\theta_1 + \theta_2)
    \left(\frac{1}{z\sin\theta_1}
    +\frac{1}{y\sin \theta_2}\right)
\end{align*}
Comparing $\d F_0(x)$
with the expression of $K(x)$ in
\eqref{eq:K_x},
we only need to prove
\begin{equation}\label{eq:ref_prove_integration}    
    \frac{2\Gamma(\frac{d}{2})}
    {\sqrt{\pi}\Gamma(\frac{d-1}{2})}
    \int_u^z (1-\frac{u^2}{x^2})^{\frac{d-3}{2}}
    |f(x,y,z)|\d x =
    2(1-\frac{u^2}{y^2})^{\frac{d-2}{2}}
    (1-\frac{u^2}{z^2})^{\frac{d-2}{2}}
\end{equation}
We can expand $f(x,y,z)$ as follows:
\begin{align*}
-f(x,y,z)&=\frac{\sin^{d-2}(\theta_2+\theta_1)
+ \sin^{d-2}(\theta_2 - \theta_1)}{z\sin\theta_1}
+\frac{\sin^{d-2}(\theta_2+\theta_1)
- \sin^{d-2}(\theta_2 - \theta_1)}{y\sin\theta_2} \\
&=\frac{2}{z\sin\theta_1}\sum_{k \textrm{ is even}}
\binom{d-2}{k} (\sin\theta_2\cos\theta_1)^{d-2-k}
(\cos\theta_2 \sin\theta_1)^k \\
&+\frac{2}{y\sin\theta_2} \sum_{k \textrm{ is odd}}
\binom{d-2}{k} (\sin\theta_2\cos\theta_1)^{d-2-k}
(\cos\theta_2 \sin\theta_1)^k
\end{align*}
Therefore,
\begin{align*}
|f(x,y,z)|
&= \frac{2x^{d-2}}{y^{d-2}z^{d-2}}
\Big[
     \sum_{k \textrm{ is even}}
     \binom{d-2}{k} (z^2-x^2)^{\frac{k-1}{2}}
     (y^2-x^2)^{\frac{d-2-k}{2}}\\
    &+ \sum_{k \textrm{ is odd}}
    \binom{d-2}{k}    (z^2-x^2)^{\frac{k}{2}}
    (y^2-x^2)^{\frac{d-3-k}{2}}
\Big]
\end{align*}
Let $a=z^2-u^2, b=y^2-u^2$, then
\begin{align*}
     &\int_u^z (1-\frac{u^2}{x^2})^{\frac{d-3}{2}}
     |f(x,y,z)|\d x
     =\frac{1}{y^{d-2}z^{d-2}}
     \int_0^a x^{\frac{d-3}{2}}\Big[ \\
     &\sum_{k \textrm{ is even}}
     \binom{d-2}{k} (a-x)^{\frac{k-1}{2}}
     (b-x)^{\frac{d-2-k}{2}}
     + \sum_{k \textrm{ is odd}}
     \binom{d-2}{k} (a-x)^{\frac{k}{2}}
     (b-x)^{\frac{d-3-k}{2}}
     \Big]\d x\\
     &=\frac{1}{2y^{d-2}z^{d-2}}
     \int_0^a x^{\frac{d-3}{2}}\Big[\frac{(\sqrt{b-x} + \sqrt{a-x})^{d-2}+(\sqrt{b-x} - \sqrt{a-x})^{d-2}}{\sqrt{a-x}}\\
     &+\frac{(\sqrt{b-x} + \sqrt{a-x})^{d-2}-(\sqrt{b-x} - \sqrt{a-x})^{d-2}}{\sqrt{b-x}}\Big] \d x
\end{align*}
Let $t=\sqrt{\frac{a-x}{b-x}}$, we then obtain
\begin{align*}
     &\int_u^z (1-\frac{u^2}{x^2})^{\frac{d-3}{2}}
     |f(x,y,z)|\d x \\
     &=\frac{(b-a)^{\frac{d-1}{2}}}{y^{d-2}z^{d-2}}\int_0^{\sqrt{a/b}}
     \left[\frac{1}{(1-t)^{d-1}}+\frac{1}{(1+t)^{d-1}}\right](a-bt^2)^{\frac{d-3}{2}}\d t
\end{align*}
Let $c=\sqrt{a/b}$ and $n=d-2\geq 0$ in the expression of $h(n,c)$ in \eqref{eq:f_n_c_expression},
we obtain
\begin{align}\label{eq:int_u_x_f_x}
     \int_u^z (1-\frac{u^2}{x^2})^{\frac{d-3}{2}}
     |f(x,y,z)|\d x
     = \frac{1}{y^{d-2}z^{d-2}}h(d-2, \sqrt{a/b}) (1-c^2)^{\frac{n+1}{2}}b^{\frac{n}{2}}
\end{align}
After simplification of \eqref{eq:int_u_x_f_x}, we obtain \eqref{eq:ref_prove_integration}.
\end{proof}

\begin{proof}[Proof of Theorem \ref{thm:H}]
     
We use $\dist(P_1,\dots, P_{d-1})$ to represent
the distance from the origin to the hyperplane passing $P_1,P_2,\dots, P_{d-1}$,
and let $F_1(x)=P(\dist(P_1,\dots, P_{d-1})\linebreak\geq x)$.

Firstly we show that $F_1(x)=K^{d-1}(x)$.
For $d=2$, it is trivial.
By induction we suppose that $F_1(x)=K^{d'-1}(x)$ is true for
$d'\leq d-1$.
Since $d'\leq d$,
the function $K(x)=P(\sqrt{(X^{(1)})^2 + (X^{(2)})^2} \geq x)$
only depends on $d$,
not on $d'$.
Therefore, in $d'$-dimensional space,
from \eqref{eq:K_x} we have $
P(\dist(P_1, \dots, P_{d'-1})\geq x) = \int_{x}^{\infty} K^{d'-2}(x)\lambda_{d'}(x/y)|\d F(y)|
$.
Then the conditional probability $P(\dist(P_1,\dots, P_{d'-1})\geq x \Big\vert |OP_1|=y)
=K^{d'-2}(x)\lambda_{d'}(x/y)$.
For $d$-dimensional space, 
we firstly specify a straight line passing through $P_{d-2}P_{d-1}$,
the space perpendicular 
to this line has $d'=d-1$ dimension while the straight line shrinks to a single point $P'$
in this subspace. Since
$P(\dist(P_1,\dots,P_{d-3},P')\geq x \Big\vert |OP'|=y)=
K^{d-3}(x)\lambda_{d-1}(x/y)$,
we have
\begin{align*}
    F_1(x) = \int_x^{+\infty} K^{d-3}(x) \lambda_{d-1}(\frac{x}{y})|\d F_0(y)|
\end{align*}
From Lemma \ref{lem:K_F_relationship} we obtain $F_1(x) = K^{d-1}(x)$.

To finish the proof and compute $H(x)$, we first specify a hyperplane $P_2P_3\dots P_{d}$,
whose distance to the origin is $y$. The subspace perpendicular to this hyperplane is a 2-D subspace,
and this hyperplane shrinks to a point $P'$ in this 2-D subspace.
Similar to (1.9) of \cite{carnal1970konvexe}, we have
\begin{align}\label{eq:H_eq_n_arccos_geometric}
     H(x) &= \frac{2}{\pi}\int_x^{+\infty}|\d F_1(y)|
     \left[ \int_x^y \arccos\frac{x}{z}|\d K(z)|
     +\int_y^{+\infty}\arccos\frac{x}{y} |\d K(z)|\right]\\
     &=\frac{2}{\pi}\int_x^{+\infty} \arccos\frac{x}{z}
     F_1(z)|\d K(z)|+\frac{2}{\pi}\int_{x}^{+\infty} K(y)F_1(y)\arccos\frac{x}{y}|\d F_1(y)| \notag 
 \end{align}
Using $F_1(x)=K^{d-1}(x)$,
we reduce \eqref{eq:H_eq_n_arccos_geometric} to
 $$
 H(x) = \frac{2}{\pi}\int_x^{+\infty}  \arccos\frac{x}{y}\cdot
 d\cdot  K^{d-1}(y) |\d K(y)|
 $$
 which is exactly \eqref{eq:H_expression_d_dim}.
 
\end{proof}
\section{Computations involved in three distribution tails}
\subsection{Computations involved in polynomial tails}
The techniques used in this subsection are similar to the derivation of $G(x)$
in (2.3) of \cite{carnal1970konvexe}.

\begin{proof}[Proof of Theorem \ref{thm:poly_tails}]
For polynomial tails defined in \eqref{eq:F_poly_tail},
we first compute the asymptotic value of $K(x)$
as follows:
\begin{align*}
     K(x) & = \int_x^{\infty} (1-\frac{x^2}{y^2})^{\frac{d-2}{2}} |\d F(y)| \\
     &= (d-2)\int_x^{\infty} F(y)\frac{x^2}{y^3} (1-\frac{x^2}{y^2})^{\frac{d-4}{2}} \d y\\
     (y=x/t) &= (d-2)x^{-k} \int_0^{1} L(x/t) t^{k+1} (1-t^2)^{(d-4)/2}\d t \\
     & \sim F(x) \frac{d-2}{2} B\left(\frac{k+2}{2}, \frac{d}{2}-1\right) 
\end{align*}
Then applying Theorem \ref{thm:H}, we obtain
\begin{align*}
     H(x) &= \frac{2}{\pi}
     \int_x^{\infty} \arccos\frac{x}{y}
     |\d K^d(y)| \\
     &\sim \frac{2}{\pi}\left(\frac{d-2}{2}\right)^d
     B^d\left(\frac{k+2}{2}, \frac{d}{2}-1\right)
     \int_x^{\infty} \frac{x}{y \sqrt{y^2-x^2}} F^d(y) \d y \\
     &\sim \frac{2}{\pi}\left(\frac{d-2}{2}\right)^d
     B^d\left(\frac{k+2}{2}, \frac{d}{2}-1\right) \frac{1}{2}
     B\left(\frac{1}{2}, \frac{dk+1}{2}\right)L^d(x) x^{-dk} \\
     &\sim \frac{2^d \pi^{(d-1)/2}\Gamma^d(\frac{k}{2}+1)
     \Gamma(\frac{dk+1}{2})}{
         \Gamma^d(\frac{k+1}{2}) \Gamma(\frac{dk}{2}+1)} G(x)^d 
\end{align*}
To simplify the notation, we denote $g(d)=\frac{2^d \pi^{(d-1)/2}\Gamma^d(\frac{k}{2}+1)
\Gamma(\frac{dk+1}{2})}{
    \Gamma^d(\frac{k+1}{2}) \Gamma(\frac{dk}{2}+1)}$.
Let $G(x)=w$, then from \eqref{eq:E_F_N_d_sim}
\begin{align}
     \E[F_N] &\sim g(d)\binom{N}{d} \int_0^{+\infty} 
      (1-w)^{N-d} |\d w^d| \notag \\
      &\sim g(d)\binom{N}{d}d\int_0^{1/2} \exp(-(N-d)w)w^{d-1}\d w
      \notag \\
      &\sim \frac{N!}{(N-d)! (N-d)^d}g(d) \label{eq:Ngd}
\end{align}
And \eqref{eq:efn_poly_second_deri} is obtained from \eqref{eq:Ngd}.
\end{proof}

The following part in this subsection discusses how to derive \eqref{eq:poly_E_F_N_d_infty}.
We first give two lemmas, which are useful to handle the case when $d\to \infty$.
\begin{lemma}\label{lem:Gamma_ratio}
     For fixed $v>0$, we have
     \begin{equation}\label{eq:Gamma_ratio}
         \frac{\Gamma(n+v)}{\Gamma(n)} \sim
         n^v
         \textrm{ as } n \to \infty             
     \end{equation}
 \end{lemma}
 \begin{proof}
     We use Beta function to prove \eqref{eq:Gamma_ratio}.
     \eqref{eq:Gamma_ratio} is equivalent to
     $B(n, v) \sim \Gamma(v) n^{-v}$.
     From the definition of beta function,
     we have
     \begin{align*}
         B(n,v) &=\int_0^1 (1-x)^{n-1} x^{v-1} \d x \\
         &\sim \int_0^{\epsilon} \exp(-nx) x^{v-1}\d x \textrm{ for given } \epsilon>0\\
         & \sim n^{-v} \int_0^{n\epsilon} \exp(-x)x^{v-1}\d x\\
         &\sim \Gamma(v) n^{-v}
     \end{align*}
 \end{proof}
We can use Stirling's formula to prove that
\begin{lemma}
\begin{equation}\label{eq:N_N_d_d}
     \frac{N!}{(N-d)! (N-d)^d} \sim 1
\end{equation}
when $N/d^2 \to \infty$.
\end{lemma}
\begin{proof}
     \begin{align*}
          \frac{N!}{(N-d)! (N-d)^d}
          &\sim \sqrt{\frac{N}{N-d}}\frac{N^N}{ e^d (N-d)^{N}}\\
          &\sim e^{-d} (1+\frac{d}{N-d})^{N} \\
          &\sim e^{-d} \exp(\frac{Nd}{N-d}) \\
          & \sim \exp(\frac{d^2}{N-d})
     \end{align*}
From the condition $N/d^2 \to \infty$,
we obtain \eqref{eq:N_N_d_d}.
\end{proof}
Combining \eqref{eq:Gamma_ratio} with \eqref{eq:N_N_d_d}, we obtain \eqref{eq:poly_E_F_N_d_infty}
from \eqref{eq:Ngd}.

\subsection{Computations involved in exponential tails}

\begin{proof}[Proof of Theorem \ref{thm:exponential_tails}]
For exponential tails, following Carnal \cite{carnal1970konvexe},
we make
the substitution $x=L(s), y=L(\sigma)$, then $\sigma=1/F(y)$.
Substituting them into \eqref{eq:K_x}, we have
\begin{align*}
    K(x)  & = \int_s^{\infty} \left(1-\frac{L(s)^2}{L(\sigma)^2} \right)^{\frac{d-2}{2}} \frac{\d \sigma}{\sigma^2} \\
\end{align*}
As $x\to \infty$, $s\to \infty$,
then $L(s)/L(\sigma) \to 1$ for $\sigma <As$,
where $A$ is an arbitrary constant.
\begin{align*}
     K(x)& \sim 2^{\frac{d-2}{2}} 
     \int_s^{As} \left(1-\frac{L(s)}{L(\sigma)}\right)^{\frac{d-2}{2}}
     \frac{\d \sigma}{\sigma^2} \\
     \textrm{ (3.6) of \cite{dwyer1991convex} }& 
     \sim  2^{\frac{d-2}{2}}  \int_s^{As} (\epsilon(s) \log\frac{\sigma}{s})^{\frac{d-2}{2}}\frac{\d \sigma}{\sigma^2}\\
     &\sim 2^{\frac{d-2}{2}} \epsilon(s)^{\frac{d}{2}-1} \frac{1}{s} \int_1^{A} (\log x )^{\frac{d-2}{2}}\frac{\d x}{x^2}\\
     & \sim 2^{\frac{d}{2}-1} \Gamma\left(\frac{d}{2}\right)F(x) v(x)^{\frac{d}{2}-1}    
\end{align*}
Similar techniques can be applied to compute the asymptotic value of $H(x)$.
\begin{align*}
    H(x) & = \frac{2}{\pi} \int_x^{+\infty}\frac{x K^d(y)}{y^2 \sqrt{1-\frac{x^2}{y^2}}}\d y\\
    &=\frac{2}{\pi}\left[2^{\frac{d}{2}-1} \Gamma\left(\frac{d}{2}\right)\right]^d
    \int_s^{+\infty}\frac{L(s) \epsilon^{1+\frac{d(d-2)}{2}}(\sigma)}{L(\sigma) \sqrt{1-\frac{L^2(s)}{L^2(\sigma)}}\sigma^{d+1}} \d \sigma
    \textrm{ since } \d y=L(\sigma)\frac{\epsilon(\sigma)}{\sigma}\d \sigma \\
    &\sim  \frac{\sqrt{2}}{\pi}
    \left[2^{\frac{d}{2}-1} \Gamma\left(\frac{d}{2}\right)\right]^d
    \epsilon^{\frac{d(d-2)+1}{2}}(s)
    \int_s^{+\infty} \left(\log\frac{\sigma}{s}
    \right)^{-\frac{1}{2}}
    \sigma^{-(d+1)}\d \sigma \\
    &= \frac{\sqrt{2}}{\sqrt{\pi d}}\left[2^{\frac{d}{2}-1} \Gamma\left(\frac{d}{2}\right)\right]^d
    \epsilon^{\frac{d(d-2)+1}{2}}(s)s^{-d}\\
    &= \frac{2^{\frac{d(d-2)+1}{2}}}{\sqrt{\pi d}}\Gamma^d\left(\frac{d}{2}\right)
    v^{\frac{d(d-2)+1}{2}}(x)F^d(x)
\end{align*}
From \eqref{eq:G_x_exp}, we obtain \eqref{eq:H_x_exp}.

Dwyer \cite{dwyer1991convex} has shown that for exponential tails, $v(x) \sim \epsilon(1/G(x))$ as $x\to \infty$.
From \eqref{eq:E_F_N_d_sim},
we use the Abel-Tauber theorem
\cite{omey1989abelian} and obtain
\begin{align*}
    \E[F_N] & \sim \frac{\pi^{\frac{d-1}{2}} 2^{\frac{d+1}{2}}}{\sqrt{d}}\int_0^{1/2} \exp(-(N-d)w) \d [w^d \epsilon(1/w)^{\frac{-d+1}{2}}] \\
    &\sim \frac{\pi^{\frac{d-1}{2}} 2^{\frac{d+1}{2}}}{\sqrt{d}} \binom{N}{d}\frac{\Gamma(d+1)}{(N-d)^d} \epsilon(N-d)^{\frac{-d+1}{2}}
\end{align*}
\end{proof}

To obtain $\E[F_N]$ when $d\to \infty$,
we need the following lemma.
\begin{lemma}\label{lem:ratio_epsilon}
     Suppose $\epsilon(x)$ is a slowly varying function
     and $N(d)\gg d^2$,
     then $\epsilon(N-d)^{d} \sim \epsilon(N)^d$.
\end{lemma}
\begin{proof}
     Since $\epsilon(x)$ is slowly varying, there exists
     a bounded function $u(x)$ such that
     $\epsilon(y)/\epsilon(x)=\exp(\int_{x}^y \frac{u(t)}{t}\d t)$.
     Besides, $\lim_{x\to \infty} u(x) = 0$.
     By the mean value theorem, there exists $\xi \in [x,y]$ such that
     $\epsilon(y)/\epsilon(x)=\exp(u(\xi)\int_{x}^y \frac{1}{t}\d t)
     =\exp(u(\xi)\log(y/x))$. Let $x=N-d, y=N$, and we obtain
     \begin{equation*}
          \frac{\epsilon(N)^d}{\epsilon(N-d)^d}
          = \exp \left(u(\xi) \log\left(\frac{N}{N-d} \right)d \right)
     \end{equation*}
     As $N-d\to \infty$, $u(\xi)\to 0$. In addition,
     $\log(\frac{N}{N-d}) \sim \frac{d}{N-d} \sim \frac{d}{N}$.
     Therefore,
     $\frac{\epsilon(N)^d}{\epsilon(N-d)^d}\sim \exp(\frac{u(\xi)d^2}{N})
     $. Since $N/d^2 \to \infty, \frac{u(\xi)d^2}{N} \to 0$ and
     $\frac{\epsilon(N)^d}{\epsilon(N-d)^d}\to 1$.
\end{proof}
By Lemma \ref{lem:ratio_epsilon}, we can show that
$\epsilon(N-d)^{\frac{-d+1}{2}}\sim \epsilon(N)^{\frac{-d+1}{2}}$.
When $d\to\infty$ and $N/d^2\to \infty$, we further obtain \eqref{eq:d_infty_exp_E_F_N}.
 
\subsection{Computations involved in truncated tails}
\begin{proof}[Proof of Theorem \ref{thm:truncated_tails}]
First we compute $K(x)$ as follows
\begin{align*}
    K(1-x)  & = \int_{1-x}^1 \left(1-\frac{(1-x)^2}{y^2}\right)^{\frac{d-2}{2}} |\d F(y)| \\
    & \sim 2^{\frac{d-2}{2}} \int_{1-x}^1 (y-1+x)^{\frac{d-2}{2}} |\d F(y)| \\
    & \sim 2^{\frac{d-2}{2}} \int_0^x (x-y)^{\frac{d-2}{2}} |\d F(1-y)| \\
    & \sim 2^{\frac{d-2}{2}} \frac{d-2}{2} \int_0^x y^k L(1/y) (x-y)^{\frac{d-4}{2}} \d y\\
    & \sim (d-2)2^{\frac{d}{2}-2} B(k+1,\frac{d}{2}-1)x^{k+\frac{d}{2}-1}L(1/x)\\
    & = k 2^{\frac{d}{2}-1} B(k, \frac{d}{2}) x^{k+\frac{d}{2}-1} L(1/x) 
\end{align*}
Then using $\arccos x \sim \sqrt{2(1-x)}$ as $x\to 0^+$,
we have
\begin{align*}
     H(1-x) & = \frac{2}{\pi} \int_{1-x}^{1}
     \arccos\frac{1-x}{y}
     |\d K^d(y)| \\
     &\sim \frac{2}{\pi}\int_0^x \sqrt{2}\sqrt{x-y} |\d K^d(1-y)| \\
     &\sim \frac{\sqrt{2}}{\pi} k^d
     2^{d(\frac{d}{2}-1)} B^d(k, \frac{d}{2})
      \int_0^x \frac{1}{\sqrt{x-y}} y^{d(k+\frac{d}{2}-1)} L^d(1/y) \d y\\
     &\sim \frac{k^d}{\pi}
     2^{\frac{1}{2} + d(\frac{d}{2}-1)} B^d(k, \frac{d}{2})
     B\left( \frac{1}{2},
     d(k+\frac{d}{2} -1)+1 \right) L^d(1/x) x^{d(k+\frac{d}{2}-1)+\frac{1}{2}}
\end{align*}
which is \eqref{eq:truncated_H_1_x}.

From \eqref{eq:truncated_G_1_x} and \eqref{eq:truncated_H_1_x}, it follows that
\begin{equation}
    H(1-x) \sim ba^{-d} G(1-x)^d x^{\frac{1-d}{2}}
\end{equation}
Let $G(1-x)=w$, similar to the analysis in (2.7) of
\cite{carnal1970konvexe}, we obtain
\begin{equation*}
     x \sim [\frac{w}{a} L(w^{-1})]^{\frac{2}{2k+d-1}} \textrm{ as } w \to 0
\end{equation*}
Finally, 
we apply the Abel-Tauber theorem
\cite{omey1989abelian}, and from \eqref{eq:E_F_N_d_sim} we obtain
\begin{align}
    \E[F_N] &= \binom{N}{d}\int_0^1 (1 - G(1-x))^{N-d} |\d H(1-x)| 
    \notag \\
    & \sim \binom{N}{d} \frac{b}{a^d} \int_0^{1/2} \exp(-(N-d)w) \d [w^d (\frac{w}{a} L(w^{-1}))^{\frac{1-d}{2k+d-1}} ]
    \notag \\
    &\sim \frac{\binom{N}{d}}{(N-d)^d} ba^{-d+\frac{d-1}{2k+d-1}}
    \Gamma 
    \left(d+1-\frac{d-1}{2k+d-1}\right)
    (N-d)^{\frac{d-1}{2k+d-1}}
    L(N-d)
    ^{\frac{d-1}{2k+d-1}}
    \label{eq:intermediate_efn_truncated}
\end{align}
When $N\to \infty$ and $d$ is a constant, we obtain \eqref{eq:efn_truncated_formula}.
\end{proof}

The asymptotic value of $\E[F_N]$ when $d\to \infty$
is derived as follows:

From \eqref{eq:a} and \eqref{eq:b}, it follows that
\begin{equation}
     ba^{-d}
     = \frac{\pi^{(d-1)/2} 2^{(d+1)/2}
     \Gamma^d(k+\frac{d+1}{2})
     \Gamma(d(k+\frac{d}{2}-1)+1)}{\Gamma^d(k+\frac{d}{2})\Gamma(d(k+\frac{d}{2}-1)+ \frac{3}{2})}
 \end{equation}
 When $d\to \infty$, using Lemma \ref{lem:Gamma_ratio},
 we obtain the asymptotic value of $ba^{-d}$ as
 \begin{equation}
     ba^{-d} \sim 2\pi^{(d-1)/2}
     d^{-1}e^kd^{d/2}
 \end{equation}
 From \eqref{eq:a}, the asymptotic value of $a$ is given as
 \begin{equation}
     a \sim \pi^{-1/2} 2^{\frac{d-2+2k}{2}} k \Gamma(k) d^{-k-\frac{1}{2}}
 \end{equation}

When $d\to \infty$ and $N/d^2\to \infty$,
from \eqref{eq:intermediate_efn_truncated}
we have $\E[F_N] \sim ba^{-d} a N^{\frac{d-1}{2k+d-1}}
L(N)
^{\frac{d-1}{2k+d-1}}$.
Then by a slight simplification,
\eqref{eq:truncated_d_inf} is obtained.
 




\fund 
\noindent This research was funded in part by the Shenzhen Science and Technology Program under Grant KQTD20170810150821146,
National Key R\&D Program of China under Grant 2021YFA0715202
and High-end Foreign Expert Talent Introduction Plan under Grant G2021032013L.

\competing 
\noindent There were no competing interests to declare which arose during the preparation or publication process of this article.



%
%
%
%

\bibliographystyle{APT}

\end{document}